\documentclass[10pt]{article}
\usepackage{graphicx}

\usepackage{amsfonts,amsmath,amssymb,amsthm}

\usepackage[margin = 3.0truecm]{geometry}
\usepackage{tikz}
\usepackage{color}
\title{A duality of mis\`{e}re games and play with a pass}
\author{Koki Suetsugu\footnote{Toyo University, suetsugu.koki@gmail.com}}
\date{}
\newtheorem{theorem}{Theorem}
\newtheorem{definition}{Definition}
\newtheorem{lemma}{Lemma}
\newtheorem{corollary}{Corollary}

\begin{document}
\maketitle
\begin{abstract}
In combinatorial game theory, the choice of play convention is a fundamental aspect of the theory.
The two most widely studied conventions are  normal play, in which the player who makes the last move wins, and  mis\`ere play, in which the player who makes the last move loses. Another well-known variant is play with a single shared pass. In such games, at most one pass may be used in total during the game: once the pass has been used, neither player may pass thereafter. Moreover, once a terminal position has been reached, passing is no longer allowed.
Recall that, for mis\`ere play, one sometimes defines SG values (or Sprague-Grundy values) in the same recursive manner as in normal play but assigns the terminal position the value 1. Also, SG values can be considered for normal-play games with a pass. 
In this work, we propose transformations that allow both mis\`ere play and play with a pass to be treated as normal-play games. As a consequence, we show that there is a certain duality between a generalization of mis\`ere play and a generalization of play with a pass.
\end{abstract}
\section{Introduction}
Combinatorial game theory studies the mathematical structure underlying two-player games with no chance moves and no hidden information. Throughout this paper, we further assume that every position has finitely many options and terminates after finitely many moves (that is, the game is {\em short}) and that the available moves from each position are the same for both players (that is, the game is {\em impartial}).

In combinatorial game theory, the choice of play convention is a fundamental aspect of the theory.
The two most widely studied conventions are  normal play, in which the player who makes the last move wins, and  mis\`ere play, in which the player who makes the last move loses. These winning conventions have been studied since the earliest days of combinatorial game theory~\cite{Bou02}, and a substantial body of research has developed around them. We refer the reader to~\cite{CGT} for further details. To make this paper self-contained, however, we provide below the definitions and terminology required for our discussion.

\subsection{SG values}
First, we introduce SG values. 
\begin{definition}
    If a legal move from $G$ leads to $G'$, then $G'$ is called an option of $G$, and we write $G \rightarrow G'$.
\end{definition}
\begin{definition}
    A position $G$ is a $\mathcal{P}$-position if the previous player (the player who has just moved) has a winning strategy.

    A position $G$ is an $\mathcal{N}$-position if the next player (the current player) has a winning strategy.
  
\end{definition}
\begin{definition}
\label{def:SGvalue}
    Let $\mathbb{N}_0$ be the set of nonnegative integers, and let ${\rm mex}(S) = \min(\mathbb{N}_0 \setminus S).$

    For every position $G$, its {\em SG value} (or Sprague-Grundy value) $\mathcal{G}(G)$ is 
    $$
    \mathcal{G}(G) = {\rm mex}(\{\mathcal{G}(G')\mid G\rightarrow G'\}).
    $$
\end{definition}
\begin{definition}
\label{def:disjunctive}
    For positions $G$ and $H$, their {\em disjunctive sum} $G+H$ is a position such that 
    $$
    \{X\mid G+H\rightarrow X\} =  \{G + H'\mid H\rightarrow H'\}\cup \{G'+H \mid G\rightarrow G'\}.
    $$  
\end{definition}
This is a recursive definition whose base is the disjunctive sum of two terminal positions, which is itself terminal.
\begin{theorem}
For every position $G$, when we play under normal play, the following holds.
\begin{eqnarray*}
\mathcal{G}(G) = 0 &\text{ if and only if}& G \text{ is a } \mathcal{P} \text{-position.} \\
\mathcal{G}(G) \neq 0 &\text{ if and only if}& G \text{ is an } \mathcal{N} \text{-position.}    
\end{eqnarray*}
\end{theorem}
\begin{theorem}
    For any positions $G$ and $H$,
    $$\mathcal{G}(G+H) = \mathcal{G}(G) \oplus \mathcal{G}(H)$$
    holds. Here, $\oplus$ is the bitwise XOR operator for binary notation (or {\sc Nim}-sum).
\end{theorem}

These results were proved independently by Sprague and Grundy \cite{Spr36, Gru39}.

To preserve the correspondence that positions with SG value $0$ are $\mathcal{P}$-positions and those with nonzero SG value are $\mathcal{N}$-positions, one sometimes considers a mis\`{e}re analogue of the SG values, in which terminal positions are assigned the value $1$.

\begin{definition}
    
    For every position $G$, its mis\`{e}re SG value $\mathcal{G}^-(G)$ is 
    $$
    \mathcal{G}^-(G) = \left\{\begin{array}{ll} 1 & \text{If } G \text{ is a terminal position} \\ {\rm mex}(\{\mathcal{G}^-(G')\mid G\rightarrow G'\}) & \text{Otherwise.} \end{array}\right.
    $$
\end{definition}

\begin{theorem}
    For every position $G$, when we play under mis\`{e}re convention, the following holds.
\begin{eqnarray*}
\mathcal{G}^-(G) = 0 &\text{ if and only if}& G \text{ is a } \mathcal{P} \text{-position.} \\
\mathcal{G}^-(G) \neq 0 &\text{ if and only if}& G \text{ is an } \mathcal{N} \text{-position.}    
\end{eqnarray*}

\end{theorem}
On the other hand, $\mathcal{G}^-(G+H) = \mathcal{G}^-(G) \oplus \mathcal{G}^-(H)$ does not always hold. 
This failure of the {\sc Nim}-sum formula is one reason why mis\`{e}re play is generally more difficult to analyze than normal play.

\subsection{Games with a pass}

Although the single shared pass convention is relatively recent compared with the normal-play and mis\`{e}re-play conventions, it has also been studied in papers like \cite{MIF16}, \cite{MFL11}.  In such games, at most one pass may be used in total during the game: once the pass has been used, neither player may pass thereafter. Moreover, once a terminal position has been reached, passing is no longer allowed.


We write $(G,1)$ when the shared pass is still available and $(G,0)$ after it has been used.
That is, $(G, 0)$ is the same position as $G$. From Definition \ref{def:SGvalue}, the SG value of position $(G, 1)$ is calculated as follows.
\begin{corollary}
For every position $G$, the SG value of $(G, 1)$ is
$$
\mathcal{G}((G, 1)) = \left\{ \begin{array}{ll} 0 & \text{If } G \text{ is a terminal position} \\ {\rm mex}(\{\mathcal{G}(G)\} \cup \{\mathcal{G}((G', 1)) \mid G\rightarrow G'\}) & \text{Otherwise.} \end{array}\right. 
$$
\end{corollary}

{\sc Nim} is a game played with one or more piles of tokens, in which a player may remove any positive number of tokens from a single pile on each turn~\cite{Bou02}. Under normal play, the SG value of a {\sc Nim} pile is equal to the number of tokens in that pile.

A game with a pass move has almost the same option structure as the disjunctive sum of the original game and a {\sc Nim} pile of size one.
The only difference is that, when a move reaches a terminal position, the additional {\sc Nim} pile can be regarded as disappearing at the same time. This exceptional behavior distinguishes a game with a pass from an ordinary disjunctive sum and makes its analysis more difficult, for example, to characterize the $\mathcal{P}$-positions of {\sc Nim} with a  pass is an open problem \cite{Now25}.  

From a similar perspective, a mis\`{e}re game in which the terminal position is assigned SG value one may be viewed in the opposite way: when a terminal position is reached, a {\sc Nim} pile of size one can be regarded as suddenly appearing. Thus, this behavior is precisely the reverse of that observed in a game with a pass.


Motivated by these observations, in the next section we introduce a framework in which games with a pass and mis\`{e}re games are regarded as dual games represented by the same underlying game tree.
Further, in Section \ref{Sec:maintheorem}, we show our main results. 

\section{Generalizations of Sprague-Grundy values}
\label{Sec:SGvalues}
Let $G$ be a position given as a game tree.
Suppose that the edges of the game tree of $G$ are partitioned into two classes. Some edges are designated as {\em ordinary}, drawn as single lines in the figures below and the remaining edges as {\em special}, drawn as double lines.
We assume that the game tree satisfies the condition that every path from the root to a leaf contains at most one special edge.
For such a game tree, together with a position of a one-pile {\sc Nim} game containing at most one token, we consider the following three rules.

Rule 1: normal play.
A legal move consists either of moving one step along the game tree, following either an ordinary edge or a special edge, or of removing the token from the one-pile {\sc Nim} component if it is present.

Rule 2: the rule based on mis\`ere play.
The game starts with no token in the one-pile {\sc Nim} component. A legal move consists of moving one step along the game tree through an ordinary edge, moving one step along the game tree through a special edge, or removing the token from the one-pile {\sc Nim} component if it is present. When a player moves along a special edge, one token is added to the one-pile {\sc Nim} component.

Rule 3: the rule based on play with a pass.
The game starts with one token in the one-pile {\sc Nim} component. A legal move consists of moving one step along the game tree through an ordinary edge, moving one step along the game tree through a special edge, or removing the token from the one-pile {\sc Nim} component if it is present. When a player moves along a special edge, the token in the one-pile {\sc Nim} component is also removed, if it is present.

{\sc Chocolate-bar games} realizing constructions analogous to Rules 1 and 3 were studied in \cite{MMS25}. Similar examples realizing Rule 2 can be constructed in the same manner.

For a position $G$ given as a game tree with ordinary edges and special edges, let $\mathcal{G}_1(G,p), \mathcal{G}_2(G,p), $ and $\mathcal{G}_3(G,p)$ be the SG values under Rules 1, 2 and 3, respectively, where $p\in \{0, 1\}$ is the number of tokens in the one-pile {\sc Nim} component.

\begin{lemma}
\label{lem:conventions}
Let $G$ be a non-terminal position.
Assume that, in the game tree of $G$, every edge leading to a terminal position is special and that no other edge is special.
When Rules 2 and 3 are played on $G$ under normal play, their SG values coincide, respectively, with the SG values of the original game without the one-pile {\sc Nim} component under mis\`ere play and under play with a pass.
\end{lemma}

\begin{proof}
    We prove that for any position $G$ given as a game tree with ordinary edges and special edges, which satisfies that every edge leading to a terminal position is special and that no other edge is special, $\mathcal{G}_2 (G, p) = \mathcal{G}^-(G)$ and $\mathcal{G}_3(G, p)= \mathcal{G}((G,p))$.

    First, we show that $\mathcal{G}_2 (G, p) = \mathcal{G}^-(G)$. Since every special edge is just before a terminal position, for a position $(G', p)$, if $p = 1$ then $G'$ is a terminal position and if it appears as an option of $(G,0)$, $\mathcal{G}_2(G', 1) = \mathcal{G}^-(G') = 1$. For any non-terminal position $G$, we use induction on the height of game tree of $G$.
    Assume that for every non-terminal option $G'$, $\mathcal{G}_2(G',0) = \mathcal{G}^-(G').$ We already have $\mathcal{G}_2(G',1) = \mathcal{G}^-(G')$ for every terminal option $G'$. Therefore, 
    \begin{eqnarray*}
    \mathcal{G}_2(G,0) &=& {\rm mex}(\{\mathcal{G}_2(G', 0)\mid G\rightarrow G' \text{ and } G' \text{ is not terminal.}\} \cup \{\mathcal{G}_2(G', 1) \mid G\rightarrow G' \text{ and } G' \text{ is terminal.}\}) \\
    &=&{\rm mex}(\{\mathcal{G}^-(G') \mid G \rightarrow G'\}) = \mathcal{G}^-(G). 
    \end{eqnarray*}

    Next, we show that $\mathcal{G}_3(G, p)= \mathcal{G}((G,p))$. For any terminal option $G'$, $\mathcal{G}_3(G',p) = \mathcal{G}_3(G', 0)= \mathcal{G}((G',p)) = 0$ since before every terminal position there is a special edge.

    For any non-terminal position $G$, we use induction on the height of game tree of $G$. Assume that for every non-terminal option $G', \mathcal{G}_3(G',p) = \mathcal{G}((G',p)).$ We already have $\mathcal{G}_3(G', p) =\mathcal{G}((G',p)) = 0$ for every terminal option $G'.$ In addition, $\mathcal{G}_3(G, 0) = \mathcal{G}((G, 0))$ holds.
    Therefore,

    \begin{eqnarray*}
        \mathcal{G}_3(G, 1) &=& {\rm mex}(\{\mathcal{G}_3(G',1)\mid G\rightarrow G' \text{ and } G' \text{ is not terminal.}\} \\ &&\cup\{\mathcal{G}_3(G',0)\mid G\rightarrow G' \text{ and } G' \text{ is terminal.}\} \cup \{\mathcal{G}_3(G,0)\})  \\
        &=&{\rm mex}(\{\mathcal{G}((G',1))\mid G\rightarrow G' \text{ and } G' \text{ is not terminal.}\} \\ &&\cup \{\mathcal{G}((G',0))\mid G\rightarrow G' \text{ and } G' \text{ is terminal.}\} \cup \{\mathcal{G}((G, 0))\}) \\
        &=& \mathcal{G}((G,1)).
    \end{eqnarray*}

\end{proof}

Figures \ref{fig:normalSG}, \ref{fig:misereSG}, and \ref{fig:passSG} show the SG values under Rules 1, 2, and 3, respectively, for a graph in which special edges occur if and only if immediately before terminal positions. 
Here we consider the state in which no move has yet been made in the one-pile {\sc Nim} component.
As Lemma \ref{lem:conventions} shows, 
Rule 2 agrees with the mis\`ere SG values, while Rule 3 agrees with the normal-play SG values for the corresponding game with a pass.

In general, it is known to be difficult to determine mis\`ere SG values or SG values for games with a pass from the normal-play SG values. However, if the positions of the special edges are changed slightly, as in Figs. \ref{fig:rule1SG}, \ref{fig:rule2SG}, and \ref{fig:rule3SG}, the SG values can be computed in a simple form.

Under Rule 2, the SG value of a vertex before a special edge is traversed, represented by a single-lined vertex, is equal to its SG value under Rule 1. The SG values of vertices after a special edge has been traversed, represented by double-lined vertices, are obtained by taking the {\sc Nim}-sum with $1$ to their SG value under Rule 1; equivalently, it is the {\sc Nim}-sum of the SG value of the original game tree and the SG value of the one-pile {\sc Nim} component. By contrast, under Rule 3, the SG value of a vertex before a special edge is traversed is obtained by taking the {\sc Nim}-sum with $1$ to its SG value under Rule 1; equivalently, it is the {\sc Nim}-sum of the SG value of the original game tree and the SG value of the one-pile {\sc Nim} component. On the other hand, the SG values after a special edge has been traversed are equal to the SG values under Rule 1. Since the token is removed from the one-pile {\sc Nim} component at the moment a special edge is traversed, the disjunctive sum structure is preserved.

\begin{figure}[htb]
    \centering
    \begin{minipage}{.25\linewidth}
        \begin{tikzpicture}[auto]
            \node[shape = circle, draw] (a) at (0, 1){2};
            \node[shape = circle, draw, double] (b) at (2, 1){0};
            \node[shape = circle, draw] (c) at (2, 0){1};
            \node[shape = circle, draw, double] (d) at (4, 0){0};
            \draw[double, ->] (a) -- (b);
            \draw[->] (a) -- (c);
            \draw[double, ->] (c) -- (d);
        \end{tikzpicture}
    \caption{SG values under normal play}
    \label{fig:normalSG}    
    \end{minipage}
    \begin{minipage}{.05\linewidth}
    \hspace{1truemm}
    \end{minipage}
    \begin{minipage}{.25\linewidth}
            \begin{tikzpicture}[auto]
            \node[shape = circle, draw] (a) at (0, 1){2};
            \node[shape = circle, draw, double] (b) at (2, 1){1};
            \node[shape = circle, draw] (c) at (2, 0){0};
            \node[shape = circle, draw, double] (d) at (4, 0){1};
            \draw[double, ->] (a) -- (b);
            \draw[->] (a) -- (c);
            \draw[double, ->] (c) -- (d);
        \end{tikzpicture}
    \caption{SG values under mis\`ere play}
    \label{fig:misereSG}    
    \end{minipage}
        \begin{minipage}{.05\linewidth}
        \hspace{1truemm}
    \end{minipage}
    \begin{minipage}{.25\linewidth}
                \begin{tikzpicture}[auto]
            \node[shape = circle, draw] (a) at (0, 1){1};
            \node[shape = circle, draw, double] (b) at (2, 1){0};
            \node[shape = circle, draw] (c) at (2, 0){2};
            \node[shape = circle, draw, double] (d) at (4, 0){0};
            \draw[double, ->] (a) -- (b);
            \draw[->] (a) -- (c);
            \draw[double, ->] (c) -- (d);
        \end{tikzpicture}
    \caption{SG values for play with a pass}
    \label{fig:passSG}    
    \end{minipage}
    
\end{figure}

\begin{figure}[htb]
    \centering
    \begin{minipage}{.25\linewidth}
        \begin{tikzpicture}[auto]
            \node[shape = circle, draw] (a) at (0, 1){2};
            \node[shape = circle, draw, double] (b) at (2, 1){0};
            \node[shape = circle, draw, double] (c) at (2, 0){1};
            \node[shape = circle, draw, double] (d) at (4, 0){0};
            \draw[double, ->] (a) -- (b);
            \draw[double, ->] (a) -- (c);
            \draw[->] (c) -- (d);
        \end{tikzpicture}
    \caption{SG values under Rule 1}
    \label{fig:rule1SG}    
    \end{minipage}
        \begin{minipage}{.05\linewidth}
    \hspace{1mm}
    \end{minipage}
    \begin{minipage}{.25\linewidth}
            \begin{tikzpicture}[auto]
            \node[shape = circle, draw] (a) at (0, 1){2};
            \node[shape = circle, draw, double] (b) at (2, 1){1};
            \node[shape = circle, draw, double] (c) at (2, 0){0};
            \node[shape = circle, draw, double] (d) at (4, 0){1};
            \draw[double, ->] (a) -- (b);
            \draw[double,->] (a) -- (c);
            \draw[->] (c) -- (d);
        \end{tikzpicture}
    \caption{SG values under Rule 2}
    \label{fig:rule2SG}    
    \end{minipage}
        \begin{minipage}{.05\linewidth}
    \hspace{1mm}
    \end{minipage}
    \begin{minipage}{.25\linewidth}
                \begin{tikzpicture}[auto]
            \node[shape = circle, draw] (a) at (0, 1){3};
            \node[shape = circle, draw, double] (b) at (2, 1){0};
            \node[shape = circle, draw, double] (c) at (2, 0){1};
            \node[shape = circle, draw, double] (d) at (4, 0){0};
            \draw[double, ->] (a) -- (b);
            \draw[double, ->] (a) -- (c);
            \draw[->] (c) -- (d);
        \end{tikzpicture}
    \caption{SG values under Rule 3}
    \label{fig:rule3SG}    
    \end{minipage}
    
\end{figure}

The theorem obtained in this paper shows that for any game tree, the good behavior of SG values under Rule 2 like Fig. \ref{fig:rule2SG} occurs if and only if the good behavior of SG values under Rule 3 like Fig. \ref{fig:rule3SG} occurs.

\section{Main theorem}
\label{Sec:maintheorem}
\begin{theorem}
    Let $G$ be a game position and let $p \in \{0, 1\}$ be the number of tokens in the one-pile {\sc Nim} component combined with $G$. Denote the SG values under Rules 1, 2, and 3 by $\mathcal{G}_1(G,p), \mathcal{G}_2(G,p), \mathcal{G}_3(G, p)$, respectively. Also, let $\mathcal{G}(G)$ denote the ordinary normal-play SG value of $G$.
    
    For a given ruleset,
    $$\mathcal{G}_1(G,p) = \mathcal{G}_2(G,p) =  \mathcal{G}(G) \oplus p $$
    holds for every position $G$ and every $p \in \{0, 1\}$ if and only if
        $$\mathcal{G}_1(G,p) = \mathcal{G}_3(G,p) =  \mathcal{G}(G) \oplus p $$
        holds for every position $G$ and every $p \in \{0, 1\}$.
\end{theorem}

\begin{proof}
We prove both implications by induction on the height of the subtree rooted at $G$. If a special edge has already been traversed on the path from the root to $G$, no special edge occurs below $G$, and all three rules reduce to the ordinary disjunctive sum of $G$ and a {\sc Nim} pile of size $p$. Hence the assertion follows immediately. We may therefore assume that no special edge has yet been traversed.



    Let 
    $$S_1(G)=\{G'\mid G\rightarrow G' \text{ along an ordinary edge}\}$$
    and
    $$S_2(G)=\{G'\mid G\rightarrow G' \text{ along a special edge}\}.$$
    Since there is no confusion, for brevity, $S_1(G)$ will be abbreviated as $S_1$, and $S_2(G)$ as $S_2$.

    First, assume that $\mathcal{G}_1(G, p)= \mathcal{G}_2(G,p) = \mathcal{G}(G)\oplus p$ holds for every position in a given ruleset. 
    
    Then, 
    \begin{eqnarray*}
      \mathcal{G}_1(G, 0) &=& {\rm mex}(\{\mathcal{G}_1(G', 0) \mid G' \in S_1\} \cup \{\mathcal{G}_1(G',0)\mid G'\in S_2\} ) \\
      &=&{\rm mex}(\{\mathcal{G}(G')\mid G'\in S_1\} \cup \{\mathcal{G}(G')\mid G'\in S_2\}),
    \end{eqnarray*}
    and
    \begin{eqnarray*} 
    \mathcal{G}_2(G,0) &=& {\rm mex}(\{\mathcal{G}_2(G',0)\mid G'\in S_1\}\cup \{\mathcal{G}_2(G',1)\mid G'\in S_2\}) \\
    &=& {\rm mex}(\{\mathcal{G}(G')\mid G'\in S_1\}\cup \{\mathcal{G}(G') \oplus 1\mid G'\in S_2\}).
    \end{eqnarray*}

    Since $\mathcal{G}_1(G, 0) = \mathcal{G}_2(G, 0) = \mathcal{G}(G), $ for any nonnegative integer $k < \mathcal{G}(G),$ 
    $$
    k \in \{\mathcal{G}(G') \mid G'\in S_1\} \text{ or } k \in \{\mathcal{G}(G')\mid G' \in S_2\} \cap \{\mathcal{G}(G')\oplus 1\mid G'\in S_2\}
    $$
    holds.

    Consider $\mathcal{G}_3(G, p)$. If $p = 0,$ then $\mathcal{G}_3(G, 0) = \mathcal{G}(G)$ since $p$ never increases under Rule 3. Let $p = 1.$

    From the induction hypothesis,
    \begin{eqnarray*}
        \mathcal{G}_3 (G, 1) &=&  {\rm mex}(\{ \mathcal{G}_3(G', 1) \mid G' \in S_1\} \cup \{\mathcal{G}_3(G', 0)\mid G' \in S_2\} \cup \{\mathcal{G}_3(G, 0)\} )\\
        &=& {\rm mex}(\{\mathcal{G}(G')\oplus1\mid G'\in S_1\} \cup \{\mathcal{G}(G')\mid G'\in S_2\}\cup\{\mathcal{G}(G)\}).
    \end{eqnarray*}

    We show that for every nonnegative integer $k < \mathcal{G}(G) \oplus 1, k \in \{\mathcal{G}(G')\oplus1\mid G'\in S_1\} \cup \{\mathcal{G}(G')\mid G'\in S_2\}\cup\{\mathcal{G}(G)\}.$ 

    Since for $k < \mathcal{G}(G),$
    $$k \in \{\mathcal{G}(G') \mid G'\in S_1\} \text{ or } k \in \{\mathcal{G}(G')\mid G' \in S_2\} \cap \{\mathcal{G}(G')\oplus 1\mid G'\in S_2\}$$
    holds and $\mathcal{G}(G) \oplus 1 \le \mathcal{G}(G) +1 ,$ there are three cases:
    \begin{itemize}

        \item If $k < \mathcal{G}(G)$ and $k \in \{\mathcal{G}(G') \mid G'\in S_2\},$ then it is clear that $ k \in \{\mathcal{G}(G')\oplus1\mid G'\in S_1\} \cup \{\mathcal{G}(G')\mid G'\in S_2\}\cup\{\mathcal{G}(G)\}.$
        \item If $k < \mathcal{G}(G), k \in \{\mathcal{G}(G')\mid G'\in S_1\}$ and $k \not\in \{\mathcal{G}(G')\mid G'\in S_2\}$, then, assume that $k \not \in \{\mathcal{G}(G') \oplus 1\mid G' \in S_1\}.$ We have $k \oplus 1 \not \in \{\mathcal{G}(G')\mid G'\in S_1\}$. Thus, $k \oplus 1 \in \{\mathcal{G}(G')\mid G'\in S_2\} \cap \{\mathcal{G}(G') \oplus 1\mid G' \in S_2\}$, which contradicts $k \not \in \{\mathcal{G}(G')\mid G'\in S_2\}$.
        \item If $k = \mathcal{G}(G), $ then it is clear that $k \in \{\mathcal{G}(G')\oplus1\mid G'\in S_1\} \cup \{\mathcal{G}(G')\mid G'\in S_2\}\cup\{\mathcal{G}(G)\}.$
    \end{itemize}
    Therefore, $k \in \{\mathcal{G}(G')\oplus1\mid G'\in S_1\} \cup \{\mathcal{G}(G')\mid G'\in S_2\}\cup\{\mathcal{G}(G)\}.$
    In addition, for $k = \mathcal{G}(G)\oplus 1,$ we have $k \not\in \{\mathcal{G}(G')\oplus1\mid G'\in S_1\} \cup \{\mathcal{G}(G')\mid G'\in S_2\}\cup\{\mathcal{G}(G)\}$ as follows:
    \begin{itemize}
        \item If $k \in \{\mathcal{G}(G')\oplus  1 \mid G'\in S_1\},$ then $\mathcal{G}(G) \in \{\mathcal{G}(G')\mid G'\in S_1\},$ which contradicts $\mathcal{G}(G) \neq \mathcal{G}(G').$
        \item If $k \in \{\mathcal{G}(G')\mid G'\in S_2\},$ then $\mathcal{G}(G) \in \{\mathcal{G}(G') \oplus 1\mid G'\in S_2\},$ which contradicts $\mathcal{G}_2(G, 0) = \mathcal{G}(G) =  {\rm mex}(\{\mathcal{G}(G')\mid G'\in S_1\}\cup \{\mathcal{G}(G') \oplus 1\mid G'\in S_2\}).$
        \item If $k = \mathcal{G}(G),$ then $k \neq \mathcal{G}(G) \oplus 1.$
    \end{itemize}
    Therefore, $$\mathcal{G}_3(G,1) = {\rm mex}(\{\mathcal{G}(G')\oplus1\mid G'\in S_1\} \cup \{\mathcal{G}(G')\mid G'\in S_2\}\cup\{\mathcal{G}(G)\}) = \mathcal{G}(G) \oplus 1.$$

    Next, assume that $\mathcal{G}_1(G,p) = \mathcal{G}_3(G,p) = \mathcal{G}(G)\oplus p$ holds for every position in a given ruleset. Then,     
    \begin{eqnarray*}
        \mathcal{G}_1(G, 1) &=& {\rm mex}(\{\mathcal{G}_1(G',1)\mid G'\in S_1\}\cup \{\mathcal{G}_1(G',1)\mid G'\in S_2\}\cup \{ \mathcal{G}_1(G,0)\} )\\ &=&{\rm mex}(\{\mathcal{G}(G')\oplus1\mid G'\in S_1\}\cup \{\mathcal{G}(G')\oplus 1\mid G'\in S_2\}\cup \{ \mathcal{G}(G)\} )
    \end{eqnarray*}
    and
    \begin{eqnarray*}
        \mathcal{G}_3(G,1)  &=& {\rm mex}(\{\mathcal{G}_3(G', 1)\mid G'\in S_1\} \cup \{\mathcal{G}_3(G',0)\mid G'\in S_2 \} \cup \{\mathcal{G}_3(G,0)\}) \\
        &=& {\rm mex}(\{\mathcal{G}(G')\oplus1\mid G'\in S_1\}\cup \{\mathcal{G}(G')\mid G'\in S_2\} \cup\{\mathcal{G}(G)\}).
    \end{eqnarray*}
    Since $\mathcal{G}_1(G,1) = \mathcal{G}_3(G,1) = \mathcal{G}(G) \oplus 1,$ for any nonnegative integer $k < \mathcal{G}(G)\oplus 1,$
    $$
    k \in \{\mathcal{G}(G')\oplus1 \mid G' \in S_1\} \text{ or } k\in \{\mathcal{G}(G')\mid G'\in S_2\} \cap \{\mathcal{G}(G')\oplus 1 \mid G'\in S_2\} \text{ or } k = \mathcal{G}(G)
    $$
    holds.

    Consider $\mathcal{G}_2(G,p).$ Since we assume that a special edge has not been traversed, we assume that $p = 0$.
    From the induction hypothesis,
    \begin{eqnarray*}
        \mathcal{G}_2(G,0) &=& {\rm mex}(\{\mathcal{G}_2(G',0)\mid G'\in S_1\} \cup \{\mathcal{G}_2(G',1)\mid G'\in S_2\}) \\
        &=& {\rm mex}(\{\mathcal{G}(G') \mid G'\in S_1\} \cup \{\mathcal{G}(G')\oplus 1\mid G' \in S_2\}).
    \end{eqnarray*}
    We show that for every nonnegative integer $k < \mathcal{G}(G), k \in \{\mathcal{G}(G') \mid G'\in S_1\} \cup \{\mathcal{G}(G')\oplus 1\mid G' \in S_2\}.$

    Since $\mathcal{G}(G)\oplus 1 \ge \mathcal{G}(G)-1,$ there are three cases:
    \begin{itemize}
        \item If $k < \mathcal{G}(G)\oplus 1$ and $k \in \{\mathcal{G}(G')\oplus 1\mid G' \in S_2\},$ then it is clear that $k \in \{\mathcal{G}(G') \mid G'\in S_1\} \cup \{\mathcal{G}(G')\oplus 1\mid G' \in S_2\}.$
        \item If $k < \mathcal{G}(G)\oplus 1, k \in \{\mathcal{G}(G')\oplus 1 \mid G'\in S_1\}$ and $k \not \in \{\mathcal{G}(G') \oplus 1 \mid G' \in S_2\},$ then, assume that $k \not \in \{\mathcal{G}(G')\mid G'\in S_1\}.$ We have $k\oplus 1 \not \in \{\mathcal{G}(G')\oplus 1\mid G'\in S_1\}.$ Thus, $k\oplus 1 \in \{\mathcal{G}(G')\mid G'\in S_2\}\cap \{\mathcal{G}(G')\oplus 1 \mid G'\in S_2\}, $ which contradicts $k \not \in \{\mathcal{G}(G') \oplus 1 \mid G' \in S_2\}.$
        \item If $k = \mathcal{G}(G) \oplus 1 < \mathcal{G}(G),$ then, $G$ has an option $H$ such that $\mathcal{G}(H) = \mathcal{G}(G)\oplus 1.$ If $H \in S_1,$ it is clear that $k\in \{\mathcal{G}(G')\mid G'\in S_1\}.$ Assume that $H\in S_2$. Since $\mathcal{G}_3(G,1) = \mathcal{G}(G)\oplus 1 = {\rm mex}(\{\mathcal{G}(G')\oplus 1 \mid G'\in S_1\} \cup \{\mathcal{G}(G')\mid G'\in S_2\} \cup \{\mathcal{G}(G)\}),$ we have $ \mathcal{G}(G)\oplus 1 \not\in \{\mathcal{G}(G')\mid G'\in S_2\},$ which is a contradiction.
    \end{itemize}
    Therefore, $k \in \{\mathcal{G}(G') \mid G'\in S_1\} \cup \{\mathcal{G}(G')\oplus 1\mid G' \in S_2\}.$ In addition, for $k = \mathcal{G}(G),$ we have $k \not\in \{\mathcal{G}(G') \mid G'\in S_1\} \cup \{\mathcal{G}(G')\oplus 1\mid G' \in S_2\}$ as follows:
    \begin{itemize}
        \item If $k\in \{\mathcal{G}(G') \mid G'\in S_1\},$ it contradicts $\mathcal{G}(G) \neq \mathcal{G}(G')$.
        \item If $k \in \{\mathcal{G}(G')\oplus 1\mid G' \in S_2\}, $ then $\mathcal{G}(G) \oplus 1\in \{\mathcal{G}(G')\mid G' \in S_2\},$ which contradicts $\mathcal{G}_3(G, 1) = {\rm mex}(\{\mathcal{G}(G')\oplus1\mid G'\in S_1\}\cup \{\mathcal{G}(G')\mid G'\in S_2\} \cup\{\mathcal{G}(G)\})$.
    \end{itemize}
    Therefore, 
    $$
\mathcal{G}_2(G, 0)  ={\rm mex} (\{\mathcal{G}(G') \mid G'\in S_1\} \cup \{\mathcal{G}(G')\oplus 1\mid G' \in S_2\}) = \mathcal{G}(G). 
    $$
\end{proof}


Thus, within this framework, the obstruction to the ordinary SG formula on the mis\`{e}re side is equivalent to the corresponding obstruction on the shared-pass side.





\begin{thebibliography}{9}
\bibitem{Bou02} C. L. Bouton, Nim, a game with a complete mathematical theory, Annals of Mathematics, {\bf 3}, pp. 35--39 (1902).
 \bibitem{Gru39} P. M. Grundy, Mathematics and games, Eureka, {\bf 32}, pp.6--8 (1939). 
 \bibitem{MIF16} R. Miyadera, M. Inoue, and M. Fukui, Impartial chocolate bar games with a pass, Integers {\bf 16}, \#G5 (2016).
\bibitem{MMS25} R. Miyadera, H. Manabe, and A. Singh, Generalizations of two-dimensional and three dimensional chocolate bar games, Integers {\bf 25} \#G3 (2025).
\bibitem{MFL11} R. E. Morrison, E. J. Friedman, and A. S. Landsberg, Combinatorial games with a pass: a dynamical systems approach, Chaos {\bf 21}(4), \#043108, (2011).
\bibitem{Now25} R. J. Nowakowski, Unsolved problems in combinatoiral games, Games of No Chance {\bf VI}, pp.55--98 (2025).
\bibitem{CGT}
A. N. Siegel, Combinatorial Game Theory, American Mathematical Society (2013).
\bibitem{Spr36} R. P. Sprague, \"{U}ber mathematische Kampfspiele, T\^{o}hoku Mathematical Journal, {\bf 41}, pp.438--444 (1935--1936).

\end{thebibliography}
\end{document}